\documentclass[12pt]{amsart}
\usepackage{amsfonts,amsmath,amsthm,oldgerm,amssymb,amscd}

\usepackage[all,cmtip]{xy}

\usepackage{fancyhdr}
\usepackage{psfrag}
\usepackage{graphicx}

\usepackage[indent]{caption}
\usepackage{indentfirst}
\usepackage{float}
\usepackage{latexsym}
\usepackage{graphics}
\usepackage{verbatim}

\newcommand{\Q}{{\mathbb Q}}

\newcommand{\C}{{\mathbb C}}

\newtheorem{theorem}{Theorem}
\newtheorem{lemma}{Lemma}
\newtheorem{corollary}{Corollary}
\newtheorem{proposition}[theorem]{Proposition}
\numberwithin{theorem}{section}
\numberwithin{corollary}{section}
\numberwithin{lemma}{section}

\theoremstyle{definition}

\numberwithin{conj}{section}

\numberwithin{example}{section}
\newtheorem{definition}{Definition}
\numberwithin{definition}{section}

\numberwithin{question}{section}
\numberwithin{equation}{section}

\theoremstyle{remark}
\newtheorem{remark}{Remark}
\numberwithin{remark}{section}

\begin{document}

\title[Locally potentially equivalent Galois representations]{Locally 
potentially equivalent Galois representations}
\date{}
\author{Vijay~M.~Patankar}
\author{C.~S.~Rajan}

\address{Tata Institute of Fundamental Research, 
Homi Bhabha Road, Bombay - 400 005, India.}
\email{ vijay.patankar@utoronto.ca, rajan@math.tifr.res.in}


\begin{abstract} We show that if two continuous semi-simple \( \ell
\)-adic Galois representations are locally potentially equivalent at a
sufficiently large set of places then they are globaly potentially
equivalent. We also prove an analogous result for arbitrarily varying
powers of character values evaluated at the Frobenius conjugacy
classes.  In the context of modular forms, we prove: given two
non-CM newforms $f$ and $g$ of weight at least two,  such that
$a_p(f)^{n_p}=a_p(g)^{n_p}$ on a set of primes of positive upper
density and for some set of natural numbers $n_p$, then $f$ and $g$
are twists of each other by a Dirichlet character.  
\end{abstract}

\maketitle

\section{Introduction}

Let $S(k,N,\epsilon)$ be the collection of newforms on $\Gamma_1(N)$
of weight $k$ and Nebentypus character $\epsilon$. Given $f\in
S(k,N,\epsilon)$ and $p$ a rational prime coprime to  $N$, let
$a_p(f)$ be the eigenvalue of $f$ with respect to the Hecke operator
at $p$. An initial motivation for this paper is to prove the following
type of recovery theorem with varying local conditions:  

\begin{theorem}\label{thm:mf} Let $f_i\in S(k_i, N_i,\epsilon_i),
~i=1,~2$ be two newforms, one of them a non-CM form of weight at least
two.   Suppose that for a collection $T$ of rational primes $p$
coprime to $N_1N_2$ of  positive upper density, there exists natural
numbers $n_p$  for $p \in T$ such that 
\[ 
a_p(f_1)^{n_p}=a_p(f_2)^{n_p}, \qquad ~~\forall ~p\in T.
\]
Then there exists a Dirichlet character $\chi$  such that 
\[ 
f_2 \simeq f_1\otimes \chi,
\]
i.e., for all $p$ coprime to $N_1N_2$,  
\[ 
a_p(f_2)= a_p(f_1)\chi(p).
\]
\end{theorem}   When $n_p=n$ is a constant and the density of the
set of places $T$ is one, the above theorem is proved in \cite{Ra2}.

In this note, we look at these  questions in the general context of 
\( \ell \)-adic semisimple representations
of the absolute Galois groups of a global field $K$. In this context,
the generalization  can be viewed from two different
perspectives. We show that \( \ell \)-adic semisimple representations
of the absolute Galois groups of a global field $K$ which are locally
potentially  equivalent at a sufficiently large set of places $K$ are
in fact globally potentially equivalent. On the other hand, we can
consider an equality of arbitrarily varying powers of the character
values evaluated at the Frobenius conjugacy classes at a sufficiently
large set of places, and deduce potential equivalence of the global
representations. 

The method of proof is as in \cite{Se, Ra1}, to look
at the algebraic monodromy groups attached to the Galois
representations, and to use algebraic versions of the Chebotarev
density theorem. This allows the use of algebraic and transcendental
methods, for example the orthogonality relation for characters of
compact groups,  to arrive at the desired conclusion. 

\section{Potentially equivalent Galois representations}

Let $K$ be a global field and let $G_K={\rm Gal}(\bar{K}/K)$ denote
the absolute Galois group over $K$ of a separable closure $\bar{K}$ of
$K$. Let $\ell$ be a rational prime coprime to the characteristic of
\( K \) and let $F$ be a $\ell$-adic local field of characteristic
zero.  Suppose $\rho: ~G_K\to GL_n(F)$ is a continuous semisimple
representation of $G_K$.  We will assume that $\rho$ is unramified
outside a finite set of places of $K$. At a place $v $ of $K$ where
$\rho$ is unramified, let $\sigma_v$ denote the Frobenius conjugacy
class in the quotient group $G_K/{\rm Ker}(\rho)$. By an abuse of
notation, we will also continue to denote by \( \sigma_v \) an element
in the associated conjugacy class.

\medskip\par

For any place $v$ of $K$, let $K_v$ denote the completion of $K$ at
$v$, and $G_{K_v}$ the corresponding local Galois group. Choosing a
place $w$ of $\bar{K}$ lying above $v$, allows us to identity
$G_{K_v}$ with the decomposition subgroup $D_w$ of $G_K$. As $w$
varies this gives a conjugacy class of subgroups of $G_K$. Given a
representation $\rho$ of $G_K$ as above, define the localization (or
the local component) \( \rho_v \) of $\rho$ at $v$, to be the
representation of $G_{K_v}$ obtained by restricting $\rho$ to a
decomposition subgroup. This is well defined upto isomorphism. 

For a finite place \( v \) of $K$, denote by \( N v \) the cardinality
of the residue field of \( K_v \). The upper density of a set \( S \)
of finite places of \( K \) is defined as:
\[ ud (S ) := \limsup_{x \to \infty } \# \{ v \in S | ~ N v \leq x \}/
\pi (x) , 
\] where \( \pi (x ) \) is the number of finite places \( v \) of \( K
\) with  \( N v \leq x \).  The set $S$ has density equal to $d(S)$ if
the limit exists as $x\to \infty$ of $\# \{ v \in S | ~ N v \leq x \}/
\pi (x)$ and is equal to $d(S)$.

\begin{definition} Suppose $\Gamma$ is an abstract group and \( \rho_i
: \Gamma \to GL_n(F) \), for \( i=1,2 \) are two linear
representations of $\Gamma$.  We say that $\rho_1$ and $\rho_2$ are
{\em potentially equivalent} if there is a subgroup $\Gamma'$ of
finite index in $\Gamma$ such that the restriction of $\rho_1$ and
$\rho_2$ to $\Gamma'$ are equivalent. 
\end{definition}

\begin{definition}
Suppose \( \rho_1 \) and \( \rho_2 \) are two Galois representations
of \( G_K \) into \( GL_n ( F ) \). We define \( \rho_1 \) and \(
\rho_2 \) to be \emph{locally potentially equivalent} at a set of
places \( T \) of \( K \), if for each \( v \in T \), the
localizations \( \rho_{1,v} \) and \( \rho_{2,v} \) are potentially
equivalent considered as representations of \( G_{K_v} \). 

\end{definition}

For the representation $\rho$ as above, let \( G \) be the algebraic 
monodromy  group attached to \( \rho \) over \( F \), i.e., 
the smallest algebraic subgroup $G$ of $M$ defined over $F$ such that
$\rho(G_K)\subset G(F)$. 
Let \( G^0 \) be the identity component of \( G \), and let \( \Phi =
G / G^0 \) be the finite group of connected components of \( G \).
Denote by \( G_i \) be the algebraic monodromy groups associated to the
 representations \( \rho_i \) for \( i = 1, 2 \), 
and let \( c_i \) be the number of  
connected components of \( G_i \) for \( i = 1, 2 \).

\medskip\par
One of the main theorems of this note is to observe that  two continuous
semisimple \( \ell \)-adic representations of a global field \( K \)
which are locally potentially equivalent at a large enough set of
places are in fact potentially equivalent:

\begin{theorem}\label{maintheorem}
Suppose $\rho_i: G_K \to GL_n(F), ~i=1,2$ are two continuous
semisimple \( \ell \)-adic representations of the absolute Galois
group $G_K$ of a global field $K$ unramified outside a finite set of
places of $K$, where \( F \) is a local field of characteristic zero
and residue characteristic \( \ell \) coprime to the
characteristic of \( K \).  Suppose there exists a set $T$ of finite
places of $K$ such that for every $v\in T$, $\rho_{1,v} $ and
$\rho_{2,v}$ are potentially equivalent. Assume that the 
upper density of \( T \), 
\[
ud (T) > min \left( 1 - \frac{1}{c_1} , 1 - \frac{1}{c_2} \right).
\]
Then, $\rho_1$ and $\rho_2$ are potentially equivalent, 
viz. there exists a finite extension $L$ of $K$ such that  
\[ 
\rho_1 \mid_{G_L} \simeq \rho_2 \mid_{G_L} .
\]
\end{theorem}

\begin{corollary}\label{cor-mt}
 With the hypothesis of Theorem \ref{maintheorem},
assume further that one of the algebraic monodromy groups, say $G_1$
is connected and the representation $\rho_1$ is absolutely
irreducible. Then there exists a Dirichlet character $\chi:~G_K\to
F^*$ such that 
\[\rho_2\simeq \rho_1\times \chi.\] 

In particular, if for every $v\in T$ there exists a character $\chi_v$
of $G_{K_v}$ such that 
\[\rho_{2,v}\simeq  \rho_{1,v}\times \chi_v, \] then there exists a
global character $\eta$ of $G_K$ such that 
\[ \rho_2\simeq  \rho_1\times \eta. \]
\end{corollary}
\begin{remark} There can be more than one $\chi_v$ satisfying
$\rho_{2,v}\simeq  \rho_{1,v}\times \chi_v$. Hence we cannot expect
that the local component of $\eta$ at $v$ be equal to $\chi_v$. 
\end{remark}

\subsection{Preliminaries} 
For the proof of Theorem \ref{maintheorem},
we recall the main results proved in \cite{Ra1}.   First, we recall Theorem 3
of \cite{Ra1}, an algebraic interpretation of results proved in
Section 6 (especially Proposition 15) of \cite{Se}, giving an
algebraic  formulation of the Chebotarev density theorem 
for the density of places satisfying an algebraic conjugacy condition:

\begin{theorem}\label{Algebraic-Chebotarev}\cite[Theorem 3]{Ra1}~~
Let \( M \) be an algebraic group defined over \( F \). Suppose 
\[
\rho : G_K \rightarrow M(F) 
\]
is a continuous representation unramified
 outside a finite set of places of $K$. 

Suppose  \( X \) is a closed subscheme of \( M\) 
defined over \( F \) and stable under the adjoint action of \( M \) 
on itself. Let 
\[
C := X(F) \cap \rho ( G_K ) .
\]
Let \( \Sigma_u \) denote the set of finite places of \( K \)
 at which \( \rho \) is unramified. 
Then the set 
\[
S : =  \{ v \in \Sigma_u  ~| ~ \rho ( \sigma_v ) \subset C \}.
\]
has a density given by 
\[
d(S) = \frac{ | \Psi | }{ | \Phi | }, 
\]
where \( \Psi = \{ \phi \in \Phi ~|~ G^\phi \subset X \} \).

\end{theorem}
\begin{remark} Since we will be able to compare only the
semi-simplifications of the representations, we assume that the
representations are semi-simple. In particular, this implies that the
algebraic monodromy group is a reductive algebraic group defined over
\( F \).  
\end{remark}

\begin{corollary}\label{semisimple} Let $\rho$ be a semisimple
continuous $l$-adic representation of $G_K$ to $GL_n(F)$ unramified
outside a finite set of places of $K$.  Then there is a density one
set of places of $K$ at which $\rho$ is unramified and the
corresponding Frobenius conjugacy class is semisimple. 
\end{corollary}

The advantage of the algebraic prescription of positive density  is
that it allows us to use techniques from the theory of reductive
algebraic groups, to change the base field and work over complex
numbers.  This allows the use of transcendental methods. 
The following result, essentially proved in Theorem 2 of
\cite{Ra1}, depends on the fact that the identity matrix is the unique
matrix in the unitary group $U(n)$ on $n$-variables having trace $n$
(see also the proof of Theorem \ref{characterpowers}): 

\begin{theorem}\label{positive-density}
Suppose \( \rho_1 \) and \( \rho_2 \) are two semisimple 
Galois representations of \( G_K \) into \( GL_n ( F ) \).
Let 
\[
\rho : = \rho_1 \times \rho_2 : G_K \to GL_n \times GL_n 
\] be the direct sum of \( \rho_1 \) and \( \rho_2 \).  Assume that
the algebraic monodromy group $G_1$  attached to \(\rho_1 \) over \( F
\) is connected.  Suppose there exists a positive density of
unramified finite places $T$ of $F$ for $\rho$, such that
\[\rm{Trace} (\rho_1(\sigma_v)) = \rm{Trace} (\rho_2(\sigma_v ))
\quad {\mbox{for}} ~ v\in T,\] where $\rho_1(\sigma_v)$
(resp. $\rho_2(\sigma_v)$) denotes the Frobenius conjugacy class at
the place $v$ associated respectively to the representations $\rho_1$
and $\rho_2$. Then  $\rho_1$ and $\rho_2$ are potentially equivalent,
i.e.,  there exists a finite extension \( L \) of \( K \) such that 
\[ \rho_1 |_{G_L} \simeq \rho_2 |_{G_L} .\]

\end{theorem}

\begin{remark} Suppose the number of connected components $G_1$ is
$c_1$. Assume further that the set of places $v\in T$ at which the
traces of the two representations are equal has upper density strictly
greater than $1-1/c_1$. By Theorem \ref{Algebraic-Chebotarev} it
follows that there is a connected component of $G$ that surjects onto
the connected component of identity in $G_1$. Arguing as above, we can
conclude that the representations are potentially equivalent. 
\end{remark}

\subsection{Proof of Theorem \ref{maintheorem}} We first observe the
following: 

\begin{lemma}\label{mainlemma} Let \( \sigma_1 \) and \( \sigma_2 \)
be two semisimple elements in \( GL_n ( F ) \) where \( F \) is local
field (finite extension of \( \Q_\ell \)). Suppose there exists a
non-zero integer \( k \) such that \( \sigma_1^k \) and \( \sigma_1^k
\) are conjugate in \( GL_n (F) \). Then, there exists a positive
integer \( m \) depending only on \( n \) and \( F \) such that
$\sigma_1^m$ and $\sigma_2^m$ are conjugate in \( GL_n (F) \).
\end{lemma} 
\begin{remark} Since we are working in \( GL_n \), two elements are
conjugate in \( GL_n ( F) \) if and only if they are  conjugate in \(
GL_n ( \bar{F} ) \). 
\end{remark}

\begin{proof} Choose an algebraic closure \( \bar{F} \) of \( F
\). Let \( F' \) be the extension of \( F \) in \( \bar{F} \)
generated by the eigenvalues of \( \sigma_1 \) and \( \sigma_2 \) in
\( \bar{F} \). It is easy to see that 
\[ [ F' : F ] \leq ( n! )^2 .
\]  The number of roots of unity contained in such a field \( F' \) is
bounded above by some positive integer \( m_0 \) depending only on  \(
[ F' : \Q_\ell ] \), thus depending only on $n$ and \( [ F : \Q_\ell ]
\). 

Let \( \{ \alpha_1, \cdots, \alpha_n \} \)  (respectively \( \{
\beta_1 , \cdots, \beta_n \} \)) be the eigenvalues of \( \sigma_1 \)
(respectively \( \sigma_2 \)). Since by our hypothesis \( \sigma_1^ k
\) is conjugate of \( \sigma_2^k \) we have up to a permutation, 
\[ \alpha_i^k  = \beta_i^k , ~ \qquad \forall ~ 1 \leq i \leq n .
\] Hence \( \alpha_i \) and \( \beta_i \) differ by a root of unity,
which lies in \( F' \).  Thus from the above comment, for \( m = m_0!
\) we have: 
\[ \alpha_i^m  = \beta_i^m , ~ \qquad  \forall ~ 1 \leq i \leq n .
\] But since both \( \sigma_1 \) and \( \sigma_2 \) are semisimple
elements in \( GL_n ( \bar{F} ) \), \( \sigma_1^m \) and \( \sigma_2^m
\) are conjugate in \( GL_n (F) \). 
\end{proof}

\medskip\par

\begin{lemma}\label{densityinextensions} Let $L/K$ be a Galois
extension of degree $d$. Suppose $T$ is a set of places of $K$ of
upper density $\delta >1-1/d$. Then the set of places $T'$ of $L$
lying above  places in $T$ is of positive upper density.  
\end{lemma}
\begin{proof} Let the upper density of $T$ be $\delta$. By Chebotarev
density theorem and hypothesis, there is a subset $T_L$ of places of
$K$ with upper density $\delta-(1-1/d)$ such that every place $v\in
T_L$  splits completely in $L$. Then the set of places $T'$ of $L$
lying above a place in $T_L$ is of positive upper density
$d(\delta-(1-1/d))$.
\end{proof}

\medskip\par We now give the proof of Theorem \ref{maintheorem}.

\begin{proof} [Proof of Theorem \ref{maintheorem}] We first observe
that by Corollary \ref{semisimple},  we can assume that \( T \)
consists of places \( v \) of \( K \) at which both \( \rho_1 \) and
\( \rho_2 \) are unramified and the corresponding Frobenius conjugacy
classes are semisimple. 

Let \( \rho := \rho_1 \times \rho_2 \). Let \( G, ~G_1 \), and \( G_2
\) be respectively the algebraic monodromy groups of \( \rho , \rho_1
\), and \( \rho_2 \). By renaming, we can assume that \( c_1 \leq c_2
\). Let \( G^0 \) and \( G_1^0 \) be respectively the connected
components of identity of \( G \) and \( G_1 \). Let \( \Phi := G /
G^0 \) be the group of connected components of \( G \). For every \(
\phi \in \Phi \), denote by \( G^\phi \) be the corresponding
connected component. 

Let $L_1$ be the finite extension of $K$ obtained by taking the
invariant field of the kernel of the homomorphism $G_K\to
G_1/G_1^0$. The degree of $L_1$ over $K$ is precisely $c_1$. Let $T'$
be the set of places of $L_1$ lying above a place in $T$ of $K$. By
Lemma \ref{densityinextensions} and the hypothesis, the upper density
of $T'$ is positive. 

Further the algebraic monodromy group of $\rho_1|_{G_{L_1}}$ is
connected, given by the connected component $G_1^0$ of $G_1$. Hence we
have reduced to the case that $G_1$ is connected. 

For every $v\in T'$, it follows that there exists an integer $n_v'$
divisible by $n_v$ (we can assume $n_v=n_v'$ by working with $v$ of
degree one over $K$) such that  \( \rho_1|_{G_{L_1}} (
{\sigma_v}^{n_v} ) \) and  \( \rho_2 |_{G_{L_1}}( {\sigma_v}^{n_v} )
\) are conjugate in \( GL_n ( F ) \). 

By Lemma \ref{mainlemma}, there exists a positive integer \( m \)
independent of \( v \in T' \), and such that for all  \( v \in T' \),
\( \rho_1 ( {\sigma_v}^m ) \) and \( \rho_2 ( {\sigma_v}^m ) \) are
conjugate in \( GL_n ( F ) \). In particular, we have 
\[ \rho ( \sigma_v ) \in X_m , ~~~ \qquad ~ \forall v \in T' , 
\] where 
\[ X_m := \{ ( g_1 , g_2 ) \in GL_n \times GL_n ~| ~ \rm{Trace} (g_1^m
) = \rm{Trace} (g_2^m ) ~ \}.
\] Now \( X_m \) is a Zariski closed subvariety of \( GL_n \times GL_n
\) invariant under conjugation. 

Since \( T' \) is of positive upper density, by Theorem
\ref{Algebraic-Chebotarev} above, there exists a connected component
\( G^\phi \) of \( G \) such that \( G^\phi \) is contained in \( X_m
\). 

Consider the map \( x \mapsto x^m \) from \( G \) to itself. Under
this map, \( G^\phi \) maps {\em onto} a connected component \( G^\psi
\) of \( G \). Since \( G^\phi \) is contained inside \( X_m \), \(
G^\psi \) is contained inside \( X_1 \). 

We now argue as in the proof of Theorem 2 of \cite{Ra1} (Theorem
\ref{positive-density} as above) to conclude that there exists a
finite extension \( L \) of \( K \) such that 
\[ \rho_1 |_{G_L} \simeq \rho_2 |_{G_L} .
\]
\end{proof}
\begin{remark} If we assume that $ud(T)=1$, then a simpler proof can
be given: we first appeal to Lemma \ref{mainlemma} and Theorem
\ref{Algebraic-Chebotarev} to conclude that $G\subset X_m$. Since the
map $x\mapsto x^m$ is surjective from $G^0$ to itself, we get that
$G^0\subset X_1$. The theorem then follows from the Chebotarev density
theorem. 
\end{remark}

\subsection{Proof of Corollary \ref{cor-mt}}
For the proof of Corollary \ref{cor-mt}, we recall  the following
proposition proved in (\cite{Ra1, Ra2}): 

\begin{proposition} \label{Schur} Suppose \( \rho_1 \) and \( \rho_2
\) are two Galois representations of \( G_K \) into \( GL_n ( F ) \)
satisfying the following: 
\begin{itemize}
\item There exists a finite extension \( L \) of \( K \) such that 
\[ \rho_1 |_{G_L} \simeq \rho_2 |_{G_L} .
\]

\item The representation $\rho_1|_{G_L}$ is absolutely irreducible
(this is guaranteed if we know that the algebraic monodromy group $G$
of $\rho$ is connected and the representation $\rho$ of $G$ is
absolutely irreducible). 
\end{itemize} Then there exists a finite order character $\chi:G_K\to
F^*$ such that 
\[  \rho_2\simeq \rho_1\otimes \chi.
\]
\end{proposition}

The proposition follows from Schur's lemma by standard
arguments. Combining the above proposition with Theorem
\ref{maintheorem}, yields a proof of Corollary \ref{cor-mt}.  

\subsection{Example: Tensor and Symmetric powers} 
We now give examples where an algebraic relation between the
representations forces potential equivalence. For a linear representation
$\rho$ of a group $\Gamma$, denote by $\chi_{\rho}$ it's
character. For a positive integer $k$, let $T^k(\rho)$
(resp. $S^k(\rho)$) denote the $k$-th tensor (resp. symmetric) power
representation associated to $\rho$.  The following theorem is proved
in \cite{Ra2}:

\begin{theorem}\label{tensor} Let  $\Gamma$ be an abstract group and
\( \rho_i:\Gamma \to GL_n(F) \) for \( i=1,2 \) be two semisimple
representations of $\Gamma$, where $F$ is a field of characteristic
zero. Suppose that  for some natural number $k$, either the $k$-th
tensor power or the symmetric power representations of $\rho_1$ and
$\rho_2$ become isomorphic. 

Then $\rho_1$ and $\rho_2$ are potentially isomorphic. 
\end{theorem}
\begin{proof} 
For the sake of completeness and also serves  to illustrate the power
of the algebraic method, we give an outline of the proof. 
   We  replace $\Gamma$ by the algebraic
envelope $G$ in $GL_n\times GL_n$ of the image of $\Gamma$ inside
$(GL_n\times GL_n) (F)$. The algebraic envelope is a reductive group,
and by going to a subgroup of finite index one can assume that $G$ is
connected. We can further take  $F$ to be the field of complex
numbers. 

For the tensor power the hypothesis implies that the $k$-th powers of
the characters are equal on $G$.  Hence the characters of the
representations differ by a $k$-th root of unity. Since the characters
are equal at identity, they are equal in a connected neighbourhood of
identity. Since a neighbourhood of identity is Zariski dense in a
connected algebraic group, and the characters are regular functions on
the group, it follows that the characters are equal on $G$. 

For the symmetric powers, the proof proceeds by imposing a
lexicographic total ordering on the set of  weights of a compact
maximal torus $D$ contained in  $G$. This is possible since the
weights are totally imaginary on the Lie algebra $\mbox{Lie}(D)$ of
$D$, and we can use the ordering on the reals.  If 
\[\lambda_1\geq\cdots \geq \lambda_n \quad (\mbox{resp.} ~\mu_1\geq
\cdots \mu_n\] are the weights of the  representation  $\rho_1$
(resp. $\rho_2$) of  $\mbox{Lie}(D)$  in $\mbox{End}(\C^n)$, the
leading $n$ weights of  the $k$-th symmetric power are given by 
\[ k\lambda_1\geq (k-1)\lambda_1+\lambda_2\cdots \geq
(k-1)\lambda_1+\lambda_n \geq \cdots\]
\[ \mbox{resp.} \quad  k\mu_1\geq (k-1)\mu_1+\mu_2\cdots \geq
(k-1)\mu_1+\mu_n \geq \cdots\] Arguing inductively allows us to say
that the set of weights occurring in the two representations of $G$
are equal, and hence $\rho_1$ and $\rho_2$ are potentially
isomorphic. 
\end{proof}
Combining the above theorem with Theorem \ref{maintheorem}, we have
the following corollary: 

\begin{corollary} Suppose $\rho_i:G_K \to GL_n(F), ~i=1,2$ are
continuous semisimple  $\ell$-adic representations of the absolute
Galois group $G_K$ of a global field $K$. Suppose there exists a set
$T$ of finite places of $K$ with $ud(T)> \mbox{min} ( 1 - 1/c_1 , 1 -
1/c_2)$ such that 
\[ R_v \circ \rho_{1,v} \simeq R_v \circ \rho_{2,v},
\] where $R_v= T^{n_v}$ or $S^{n_v}$ is either a $n_v$-th tensor or
symmetric power representation for some natural number $n_v$, \( n_v
\) depending on \( v \). Then there is a finite extension $L$ of $K$
such that 
\[  \rho_1 \mid_{G_L} \simeq \rho_2 \mid_{G_L} .
\]
\end{corollary}

\section{Powers of character values of Frobenius classes} 
In Theorem \ref{maintheorem}, we worked with equality of conjugacy 
classes. We now formulate a result that will work with character values.

\begin{theorem}\label{characterpowers} With notation as in Theorem
\ref{maintheorem},  assume that there exists a set $T$ of finite
places of $K$  such that for every $v\in T$, there exists non-zero
integers  $m_v>0$ satisfying the following: 
\[ \chi_{\rho_1}(\sigma_v)^{m_v} =\chi_{\rho_2}(\sigma_v)^{m_v}
\]  Assume that the  upper density of \( T \) satisfies the
inequality, 
\[ ud (T) > min \left( 1 - \frac{1}{c_1} , 1 - \frac{1}{c_2} \right), 
\]  Then, $\rho_1$ and $\rho_2$ are potentially equivalent, viz. there
exists a finite extension $L$ of $K$ such that  
\[  \rho_1 \mid_{G_L} \simeq \rho_2 \mid_{G_L} .
\]
\end{theorem} 

For a fixed place, the hypothesis is on a single power
of the trace of  the conjugacy class,  and thus does not necessarily
imply local equivalence. However the global arguments as in the proof
of the above theorem helps us in proving this theorem. 

\begin{proof} Arguing as in the proof of the Theorem
\ref{maintheorem}, we can assume that $G_1$ is connected.  If
$\chi_{\rho_1}(\sigma_v)$ vanishes, then the hypohesis holds for any
integer $m_v$. On the other hand, if $\chi_{\rho_1}(\sigma_v)$ is
non-zero, then $\chi_{\rho_1}(\sigma_v)$  and
$\chi_{\rho_2}(\sigma_v)$ differ by a root of unity belonging to
$F$. Since the group of roots of unity in the non-archimedean local
field $F$ is finite,   there is an integer $m$ independent of $v$,
such  that for $v\in T$, 
\[ \chi_{\rho_1}(\sigma_v)^{m} =\chi_{\rho_2}(\sigma_v)^{m}
\]  Let 
\[ X^m := \{ ( g_1 , g_2 ) \in GL_n \times GL_n ~| ~ \rm{Trace}
(g_1)^m  = \rm{Trace} (g_2)^m  ~ \}.
\] \( X^m \) is a Zariski closed subvariety of \( GL_n \times GL_n \)
invariant under conjugation.  By Theorem \ref{Algebraic-Chebotarev},
the density condition on $T$ implies the existence a connected
component $G^{\phi}\subset X^m$.  By working over the complex numbers
and with a maximal compact subgroup $J$ of $G$, we can assume that
there is an element of the form $(1,y)\in J^{\phi}\cap X^m$. Since the
only elements in an unitary group $U(n)$ with the absolute value of
it's trace being precisely $n$ are scalar matrices $\zeta I_n$ with
$|\zeta|=1$, we conclude that  $y$ is of the form $\zeta I_n$ for
$\zeta$ a $m$-th root of unity.

We can write the connected component $G^{\phi}= G^0.(1,\zeta~I_n)$. In
particular, every element $(u_1, u_2)\in G^0$ the identity  component
of $G$, can be written as 
\[(u_1, u_2)=(z_1, \zeta^{-1}z_2), \] where $(z_1, z_2)\in
G^{\phi}\cap X_m$. Since  $\zeta$ is a $m$-th root of unity, we have  
\[
\mbox{Trace}(u_1^m)=\mbox{Trace}(z_1^m)=\mbox{Trace}(z_2^m)
=\mbox{Trace}((\zeta^{-1}z_2)^m)
=\mbox{Trace}(u_2^m).
\] Hence $G^0\subset X^m$. 

We are now in the situation of Theorem \ref{tensor}.   Let $p_i,
~i=1,~2$ be the two projections from $G^0$ to $GL(n)$.  The statement
$G^0\subset X^m$ can be reformulated as saying that 
\[ \chi_{p_1}^m=\chi_{p_2}^m, \] restricted to $G^0$.  Since $G^0$ is
connected, it follows from Theorem \ref{tensor}, that the
representations $p_1$ and $p_2$ are equivalent restricted to $G^0$.
Hence it follows that $\rho_1$ and $\rho_2$ are potentially
eqiuvalent. 

\end{proof}

\begin{remark} If we just assume that $ud(T)=1$, then the proof can be
given as follows: we appeal to Theorem \ref{Algebraic-Chebotarev},
base change to the extension $L$ as in the proof of the theorem and
then invoke  Theorem \ref{tensor} to complete the proof. 
\end{remark} 

Again, combining with Proposition \ref{Schur}, we have
the following corollary: 

\begin{corollary} \label{charschur} Suppose that the algebraic
monodromy group  $G_1$ is connected and the representation $\rho_1$ is
absolutely irreducible. Assume that there exists a set of unramified
places $T$ for $\rho$ of positive density and a collection of positive
integers $m_v$ for $v\in T $, such that 
 \[ \chi_{\rho_1}(\sigma_v)^{m_v} =\chi_{\rho_2}(\sigma_v)^{m_v}
\qquad \forall v\in T.
\]  Then there exists a Dirichlet character $\chi:~G_K\to F^*$ such
that 
\[\rho_2\simeq \rho_1\times \chi.\] 
\end{corollary}

\section{An algebraic density result}
 We have the following density
result, analogous to the Chebotarev density theorem that the Frobenius
conjugacy classes are Zariski dense in the algebraic monodromy group: 

\begin{theorem} Let \( \rho : G_K \rightarrow GL_n ( F ) \) be a
continuous semisimple \( \ell \)-adic representation with connected
algebraic  monodromy group \( G \). Let \( T \) be a set of unramified
places for \( \rho \) with  positive upper density.  For \( v \in T
\),  let \( n_v \) be a positive integer. Then, the smallest algebraic
subgroup of \( G \) generated by  \( \rho ( \sigma_v )^{n_v} \), where
\( \sigma_v \) is the Frobenius conjugacy class associated to \( v \),
as \( v \) ranges over  \( T \) is \( G \).
\end{theorem}

\medskip\par

\begin{proof} Since \( \rho \) is semisimple, \( G \) is a reductive
algebraic subgroup of \( GL_n \) over \( F \). Let \( G_T \) be the
smallest algebraic subgroup of \( G \) generated by \( \rho ( \sigma_v
)^{n_v} \), where \( \sigma_v \) is the Frobenius conjugacy class
associated to \( v \), as \( v \) ranges over \( T \). Then \( G_T \)
is a closed, normal subgroup of \( G \). Choose an embedding of \( G /
G_T \) as an algebraic subgroup of \( GL_m \) over \( F \) for some
positive integer \( m \). Denote by \( \rho_1 \) the map \( \rho : G_K
\rightarrow G \) followed by quotient to the reductive algebraic group
\( G/G_T \subset GL_m \).  Let \( \rho_2 = Id : G_K \rightarrow GL_m
\) be the \emph{identity} map sending every element of \( G_K \) to
the identity element of \( GL_m \). Applying Theorem \ref{maintheorem}
to \( \rho_1 \) and \( \rho_2 \), it follows that \( \rho_1 \) and \(
\rho_2 \) are potentially equivalent. Thus, for some finite extension
\( L \) of \( K \), 
\[ \rho_1 |_{G_L} \simeq Id |_{G_L} .
\] Thus, \( \rho_1 ( G_L ) = \{ 1 \} \). But since \( \rho_1 ( G_K )\)
is Zariski dense in  the connected algebraic group \(G/G_T \), we get
\( G /G_T = \{ 1 \} \).
\end{proof}

\begin{remark} It is tempting to conjecture that we can combine
Theorems \ref{maintheorem} and \ref{characterpowers} into one single
theorem, viz.,  to be able to conclude from an equality of character
values of the form
\[ \chi_{\rho_1}(\sigma_v^{m_v})^{k_v}
=\chi_{\rho_2}(\sigma_v^{m_v})^{k_v} \qquad \forall v\in T.
\]  at a sufficiently large set of places $T$, that $\chi_1$ and
$\chi_2$ are potentially equivalent. However this would require in the
statement of the foregoing corollary, that the collection of conjugacy
classes \( \{\rho ( \sigma_v )^{n_v}\} \) is Zariski dense in $G$. It
is possible that such a statement holds in the context of $l$-adic
representations, but this is not a general statement about Zariski
dense subsets in connected groups.  For example, the set
$\{\sigma_n=1/n\}$ is Zariski dense in the affine line $G_a$, but if
we consider arbitrary multiples $\{n\sigma_n=1\}$ then it is no longer
Zariski dense in $G_a$.  
\end{remark}

\section{Modular forms} We now indicate briefly the proof of
Theorem \ref{thm:mf} stated in the introduction.  For a newform $f$ of
weight $k$, let 
\[\rho_f: G_{\Q}\to GL_2(\Q_l),\] be the $l$-adic representations of
$G_{\Q}$ associated by the work of Shimura, Ihara and Deligne.  It has
been shown by Ribet in \cite{Ri}, that the representation $\rho_f$ is
semisimple. Further if $f$ is a non-CM form of weight at least two,
then  the Zariski closure $G_f$ of the image $\rho(G_{\Q})$ is $GL_2$.
Theorem \ref{thm:mf} follows now from Corollary \ref{charschur}.

\begin{remark}  A similar statement can be made for the class of
Hilbert modular forms too. 
\end{remark}

\noindent {\bf Acknowledgements}. The authors thank Dipendra Prasad
for useful discussions and a comment which simplified vastly an
earlier proof of Lemma \ref{mainlemma} based on a purity
assumption. Some aspects related to this work were carried out when
the second author was visiting Universit\'{e} de Paris 7 in May of
2009. The second author thanks Cefipra for sponsoring the visit.

\end{document}